\documentclass{amsart}
\usepackage{amsmath,amssymb,amscd,latexsym}
\usepackage{enumerate,verbatim,calc}
\usepackage[all]{xy}
\SelectTips{eu}{}

\usepackage[colorlinks]{hyperref}

\newcommand{\ts}{\textstyle}

\newcommand{\sss}{\scriptscriptstyle}

\newcommand{\ges}{{\sss\geqslant}}

\newcommand{\ann}{\operatorname{Ann}}

\newcommand{\bwedge}{{\textstyle\bigwedge}}
\newcommand{\coker}{\operatorname{Coker}}

\newcommand{\bsx}{{\boldsymbol x}}

\newcommand{\fm}{{\mathfrak m}}
\newcommand{\fn}{{\mathfrak n}}
\newcommand{\fp}{{\mathfrak p}}
\newcommand{\fq}{{\mathfrak q}}

\newcommand{\Ass}{\operatorname{Ass}}

\newcommand{\col}{\colon}

\newcommand{\dd}{\partial}
\newcommand{\depth}{\operatorname{depth}}

\newcommand{\gdim}{\operatorname{G-dim}}
\newcommand{\grade}{\operatorname{grade}}
\newcommand{\bigrade}[2][{}]{\operatorname{bigrade}({#2}_{#1})}
\newcommand{\fd}{\operatorname{fd}}
\newcommand{\height}{{\operatorname{height}}}

\newcommand{\hcoh}[4][*]{\operatorname{HH}^{#1}(#3\var #2;#4)}

\newcommand{\env}[2][{}]{{#2}{}^{\mathsf e}_{#1}}
\newcommand{\envv}[2]{{#1}_{#2}^{\mathsf e}}

\newcommand{\hh}[1]{\operatorname{H}(#1)}

\newcommand{\HH}[2]{\operatorname{H}_{#1}(#2)}
\newcommand{\CH}[2]{\operatorname{H}^{#1}(#2)}

\newcommand{\id}{\operatorname{id}}

\newcommand{\Ker}{\operatorname{Ker}}

\newcommand{\length}{\operatorname{length}}

\newcommand{\lra}{\longrightarrow}
\newcommand{\Max}{\operatorname{Max}}
\newcommand{\Min}{\operatorname{Min}}

\newcommand{\pd}{\operatorname{pd}}

\newcommand{\rank}{\operatorname{rank}}

\newcommand{\Shift}{\mathsf{\Sigma}}
\newcommand{\Spec}{\operatorname{Spec}}
\newcommand{\Supp}{\operatorname{Supp}}

\newcommand{\Ext}[4]{\operatorname{Ext}^{#1}_{#2}(#3,#4){}}
\newcommand{\Hom}[3]{\operatorname{Hom}_{#1}(#2,#3)}
\newcommand{\Rhom}[3]{\operatorname{\mathsf{R}Hom}_{#1}(#2,#3)}
\newcommand{\dtensor}[1]{\otimes^{\mathsf{L}}_{#1}}

\newcommand{\dcat}[1][R]{{\mathsf D}(#1)}

\newcommand{\trdeg}{\operatorname{tr\,deg}}
\newcommand{\rtd}[2]{\operatorname{tr\,deg}_{#1}k({#2})}
\newcommand{\var}{{\hskip.7pt\vert\hskip.7pt}}

\newcommand{\vf}{{\varphi}}

\newcommand{\xra}{\xrightarrow}

\newcommand{\ZZ}{\operatorname{Z}}

\newcommand{\BN}{{\mathbb N}}
\newcommand{\BQ}{{\mathbb Q}}
\newcommand{\BZ}{{\mathbb Z}}

\theoremstyle{plain}

\newtheorem{theorem}{Theorem}[section]
\newtheorem{proposition}[theorem]{Proposition}
\newtheorem{lemma}[theorem]{Lemma}
\newtheorem{corollary}[theorem]{Corollary}

\newtheorem{itheorem}{Theorem}

\theoremstyle{definition}

\newtheorem{example}[theorem]{Example}

\newtheorem{chunk}[theorem]{}

\theoremstyle{remark}

\newtheorem{remark}[theorem]{Remark}
\newtheorem*{Claim}{Claim}

\newtheorem*{Conjecture}{Conjecture}

\numberwithin{equation}{theorem}

\begin{document}

\title[Gorenstein algebras] {Gorenstein algebras and Hochschild cohomology}

\author{Luchezar L.~Avramov} \address{Department of Mathematics, University of Nebraska,
  Lincoln, NE 68588, U.S.A.}  \email {avramov@math.unl.edu}

\author{Srikanth B.~Iyengar} \address{Department of Mathematics, University of Nebraska,
  Lincoln, NE 68588, U.S.A.}  \email{iyengar@math.unl.edu}

\thanks{Research partly supported by NSF grants DMS 0201904 (LLA) and DMS 0602498 (SI)}



\date{\today}

\begin{abstract}
  Links are established between properties of the fibers of a flat,
  essentially of finite type homomorphism of commutative rings
  $\sigma\colon K\to S$, and the structure of the graded $S$-module
  $\operatorname{Ext}_{S\otimes_KS}(S,{S\otimes_KS})$.  It is proved 
  that this module is invertible if and only if all fiber rings of $\sigma$ 
  are Gorenstein.  It is also proved that if $K$ is Gorenstein and $S$ 
  is Cohen-Macaulay and generically Gorenstein, then the vanishing of 
  $\operatorname{Ext}^n_{S\otimes_KS}(S,{S\otimes_KS})$ for $\dim S$ 
  consecutive specified values of $n$ is necessary  and sufficient for 
  $S$ to be Gorenstein.
\end{abstract}

\dedicatory{To Mel Hochster on his 65th birthday.}

\maketitle

\section*{Introduction}

Each one of the main classes of commutative noetherian rings---regular,
complete intersection, Gorenstein, and {Cohen-Macaulay}---is defined
by \emph{local} properties that require verification at \emph{every}
maximal ideal.  It is therefore important to develop for these properties
\emph{global} recognition criteria involving only \emph{finitely many}
invariants.  Finitely generated algebras over fields provide the test
case.  Our goal is to develop finitistic global tests applicable also
in a more general, relative situation.

To fix notation, let $K$ be a commutative noetherian ring and
$\sigma\col K\to S$ a \emph{flat} homomorphism of rings, which is
\emph{essentially of finite type}.  It is said to be {Cohen-Macaulay}
or Gorenstein if its fiber rings have the corresponding property; see
Section \ref{Quasi-Gorenstein homomorphisms} for details.  The result
below is taken from Theorem~\ref{defs:cm2}; recall that $\grade_PS$
is the smallest integer $n$ with $\Ext n{P}S{P}\ne0$.

  \begin{itheorem}
Assume that $\Spec S$ is connected.  Let $K\to P\to S$ be a factorization
of $\sigma$ with $P$ a localization of some polynomial ring and $S$
a finite $P$-module.

The map $\sigma$ is {Cohen-Macaulay} if for $g=\grade_PS$ one has
\[
\Ext nPSP=0 \quad\text{for}\quad g< n\le g+d\,.
\]
Conversely, if $\sigma$ is {Cohen-Macaulay}, then $\Ext nPSP=0$ holds for $n\ne g$.

The homomorphism $\sigma$ is Gorenstein if and only if it is Cohen-Macaulay and the $S$-module 
$\Ext gPSP$ is invertible.  
   \end{itheorem}

Thus, it is easy to recognize {Cohen-Macaulay} maps, because they
are characterized in terms of \emph{vanishing} of cohomology in a
\emph{finite} number of \emph{specified} degrees; this can be decided
from finite constructions.  On the other hand, the condition
needed to identify the Gorenstein property involves the \emph{structure}
of $\Ext gPSP$ as a module over $S$, which is not determined by finitistic
data; see Remark \ref{action}.

Partly motivated by recent characterizations of regular homomorphisms,
recalled in Section \ref{Gorenstein homomorphisms}, we approach the
problem by studying the homological properties of $S$ as a module over
the \emph{enveloping algebra} $\env S=S\otimes_KS$, which acts on $S$
through the \emph{multiplication map} $\mu\col S\otimes_KS\to S$.
As usual, $\dim S$ denotes the \emph{Krull dimension} of $S$.
For a prime ideal $\fq$ of  $S$ we let $\trdeg_K(S/\fq)$ denotes the
\emph{transcendence degree} of the field of fractions of $S/\fq$ over
that of the image of $K$.

The next theorem characterizes Gorenstein homomorphisms in terms of
properties of Ext modules over $S\otimes_KS$, even without assuming that
$\sigma$ is {Cohen-Macaulay}.

\begin{itheorem}
  \label{ihoch}
The homomorphism $\sigma\col K\to S$ is Gorenstein if and only if the
$S$-module $\bigoplus_{n=0}^\infty\Ext n{\env S}S{\env S}$ is invertible.
 \end{itheorem}

This result is contained in Theorem~\ref{diagonal:gor}, whose proof hinges
on properties of quasi-Gorenstein homomorphisms, defined and studied
in \cite{AF:qG}, and on a strengthening of Foxby's criterion for finite
G-dimension, obtained in Theorem~\ref{foxby}.  When the $K$-module $S$
is \emph{projective}, rather then just flat, the  theorem can be stated
in terms of \emph{Hochschild cohomology} by using the isomorphisms
of $S$-modules
\[
\hcoh[n]KSN\cong\Ext n{\env S}SN \quad\text{for all}\quad n\in\BZ\,.
\]

Two aspects of Theorem \ref{ihoch} present difficulties in applications:
All modules $\Ext n{\env S}S{\env S}$ are involved, and conditions other
than vanishing are imposed on these modules.  The next result identifies
critical values of $n$.

\begin{itheorem}
  \label{ithm:bigrade}
For every minimal prime ideal $\fq$ of $S$ one has
  \[
\Ext {t}{\env S}S{\env S}_\fq\ne0 
  \quad\text{with}\quad 
t=\trdeg_K(S_\fq/\fq S_\fq)\,.
  \]

If $\Spec S$ is connected and $\sigma$ is {Cohen-Macaulay}, then $t$
is independent of $\fq$.

If, moreover, $\sigma$ is Gorenstein, then $\Ext {n}{\env S}S{\env S}=0$
for $n\ne t$.
  \end{itheorem}

Theorem~\ref{ithm:bigrade} is abstracted from Theorems \ref{thm:bigrade},
\ref{big:cm} and \ref{gorenstein-factorization}.  Their proofs hinge
on an expression for the modules $\Ext n{\env S}S{\env S}$ in terms of
cohomology computed over the ring $S$.  The relevant formula complements
the classical technique of reduction to the diagonal; it is proved in
\cite{AIL} and is reproduced in \ref{computation:smooth}.

We expect that the last statement of Theorem~\ref{ithm:bigrade}
admits a strong converse:

  \begin{Conjecture}
When $\Spec S$ is connected, $\sigma$ is {Cohen-Macaulay}, and one has
  \[
\Ext n{\env S}S{\env S}=0 
   \quad\text{for}\quad 
\trdeg_K(S/\fq)<n\le\trdeg_K(S/\fq)+\min\{\dim S,1\}\,,
  \]
then the homomorphism $\sigma$ is Gorenstein.
  \end{Conjecture}

When $K$ is a field and $\rank_KS$ is finite one has $\dim S=\trdeg_K(S/\fq)=0$, 
so the preceding statement specializes to a conjecture of Asashiba's \cite[\S3]{As}, 
which strengthens the still open commutative case of a conjecture of Tachikawa:  If 
$\Ext n{\env S}S{\env S}=0$ holds for \emph{every} $n\ge1$, then $S$ 
is self-injective; see \cite[p.~115]{Tc}.  The next result is new even when 
$K$ is a field and $S$ is reduced; it is proved as 
Theorem~\ref{portmanteau:Gor}.

\begin{itheorem}
  \label{iGorenstein}
The conjecture above holds when the ring $K$ is Gorenstein and the ring 
$S_\fq$ is Gorenstein for every minimal prime ideal $\fq$ of $S$.
  \end{itheorem}

The additional hypotheses allow one to apply a recent characterization of
Gorenstein rings from \cite{ABS} or \cite{HH}, in order to prove that
the ring $S$ is Gorenstein.  Classical properties of flat maps then
imply that the homomorphism $\sigma$ is Gorenstein.

\section{Gorenstein dimension}
\label{Gorenstein dimension}

In this paper rings are commutative.  For modules \emph {finite} means
finitely generated.  The results in this section concern an invariant
of complexes called G-dimension (for \emph{Gorenstein} dimension),
defined for finite modules by Auslander and Bridger~\cite{AB}.

Let $R$ be a commutative ring. We write $\dcat$ for the derived category of
$R$-modules. It objects are the complexes of $R$-modules of the form
\[
M=\quad\cdots \to M_{n+1} \xra{\dd_{n+1}} M_n\xra{\ \dd_n\ } M_{n-1}\to \cdots
\]
We write $M\xra{\simeq} N$ to flag a \emph{quasi-isomorphism}; that
is, a morphism of complexes inducing an isomorphism in homology.
The notation $M \simeq N$ means that $M$ and $N$ are linked by a chain
of quasi-isomorphisms; that is, they are isomorphic in $\dcat$.

A complex $M$ is \emph{homologically bounded} if $\HH iM=0$ for $|i|\gg0$;
it is said to be \emph{homologically finite} if, in addition, each
$R$-module $\HH iM$ is finite.

\begin{chunk}
\label{G-modules}
Let $R$ be a noetherian ring.

An $R$-module $G$ is said to be \emph{totally reflexive} if it satisfies
\begin{alignat*}{3}
    G&\cong&\Hom R{&\Hom RGR}R \quad &&\text{and}\\
    \Ext nRGR=0&=&\Ext nR{&\Hom RGR}R\quad &&\text{for all}\quad n\ge1\,.
\end{alignat*}
The \emph{G-dimension} of a homologically finite complex $M$ of $R$-modules 
is the number
\[
  \gdim_RM=\inf_n\left\{n\ge\sup\hh M \left|\,
      \begin{gathered}
        \coker(\dd^P_{n+1})\text{ is totally reflexive in some} \\
        \text{ semiprojective resolution }P\xra{\simeq}M
      \end{gathered} \right\}\right.\,.
\]
Finite projective modules are totally reflexive, so one has $\gdim_RM\le \pd_RM$.
Every finite $R$-module (equivalently, homologically finite complex)
has finite $G$-dimension if and only if the ring $R$ is Gorenstein;
see \cite[(4.20)]{AB}.
  \end{chunk}

\begin{chunk}
\label{gd=reflexive}
Foxby~\cite[(2.2.3)]{Ch} obtained an alternative characterization of complexes of
finite G-dimension: A homologically finite complex of $R$-modules $M$ has finite
$G$-dimension if and only the following two conditions hold.
\begin{enumerate}[\quad\rm(a)]
\item The canonical \emph{biduality map} in $\dcat$ is an isomorphism:
\[
    \delta^M\col M\xra{\simeq} \Rhom R{\Rhom RMR}R\,.
\]
\item The complex $\Rhom RMR$ is homologically bounded.
\end{enumerate}
Moreover, when these conditions hold one has
\begin{equation}
    \label{eq:gd=reflexive}
    \gdim_RM = \sup\{n\mid \Ext nRMR\ne0\}\,.
\end{equation}
Thus, a lower bound for $\gdim_RM$ is provided by the \emph{grade} of $M$,
defined by
\begin{equation}
\label{grade}
        \grade_RM = \inf\{n\mid \Ext nRMR\ne 0\}\,.
\end{equation}
When $M$ is a module its grade is equal to the maximal length of an
$R$-regular sequence contained in $\ann_RM$, its annihilator ideal.
 \end{chunk}

We prove a more flexible version of Foxby's characterization of finite
Gorenstein dimension. We write $\Max R$ for the set of maximal ideals
of $R$.

\begin{theorem}
\label{foxby}
Let $R$ be a noetherian ring and $M$ a complex with $\hh M$ finite.

The complex $M$ has finite $G$-dimension when the following conditions hold:
\begin{enumerate}[\quad\rm(a)]
\item For each maximal ideal $\fm$ in $R$, in $\dcat [R_\fm]$ there exists an isomorphism
\[
  M_\fm \simeq \Rhom {R_\fm}{\Rhom {R_\fm}{M_\fm}{R_\fm}}{R_\fm}\,.
\]
\item The complex $\Rhom RMR$ is homologically bounded or $\dim R$ is finite.
\end{enumerate}
\end{theorem}

As a corollary we show that for certain modules the finiteness of
G-dimension can be read off their cohomological invariants.

\begin{corollary} \label{invertible} Let $M$ be a finite $R$-module.

If for each $\fm\in\Max R$ there exists an integer $d(\fm)$, such that one has
\[
 \Ext n{R_{\fm}}{M_{\fm}}{R_{\fm}}\cong
\begin{cases}
  M_{\fm} &\text{for}\quad n=d(\fm)\,;  \\
  0 &\text{for}\quad n\ne d(\fm)\,,
\end{cases}
\]
then the following inequalities hold:
\[
\gdim_RM \le\sup\left\{\height\, \fp \,\left|
\begin{gathered}
 \fp\text{ a prime ideal in }R \\
            \text{minimal over }\ann_RM
\end{gathered}
   \right\}\right.<\infty\,.
\]
\end{corollary}

\begin{proof}
  The hypothesis implies $\Rhom{R_\fm}{M_\fm}{R_\fm}\simeq\Shift^{d(\fm)}M_{\fm}$, and
  hence one has
\begin{align*}
        \Rhom{R_\fm}{\Rhom{R_\fm}{M_\fm}{R_\fm}}{R_\fm}
        &\simeq\Rhom{R_\fm}{\Shift^{d(\fm)}M_\fm}{R_\fm}\\
        &\simeq\Shift^{-d(\fm)}\Rhom{R_\fm}{M_\fm}{R_\fm}\\
        &\simeq M_\fm
\end{align*}
in the derived category $\dcat[R_\fm]$.  In particular, setting
\[
      h=\max\{\height\,\fp\mid \fp\in\Spec R \text{ is minimal over }\ann_RM \}\,,
\]
from \eqref{grade} and Krull's Principal Ideal Theorem one obtains inequalities
\[
      d(\fm) = \grade_{R_\fm}M_\fm \le \height (\ann_{R_\fm}M_\fm) \le h\,.
\]
They imply $\Ext nRMR=0$ for $n>h$, so $\hh{\Rhom RMR}$ is bounded. Now
one gets $\gdim_RM<\infty$ from Theorem~\ref{foxby}, and then $\gdim_RM\le
h$ from \eqref{eq:gd=reflexive}.
 \end{proof}

The proof of the theorem uses a result of independent interest:

\begin{proposition}
\label{finite:test}
Let $R$ be local ring and $X$ a complex with each $R$-module $\HH
iX$ finite.  If $\HH i{\Rhom RXR}=0$ for $|i|\gg0$, then $\HH iX=0$
for $i\ll 0$.
 \end{proposition}

\begin{proof}
Let $k$ be the residue field of $R$.  As $\HH i{\Rhom RXR}=0$ for $i\gg 0$ one has
\[
       \HH n{\Rhom Rk{\Rhom RXR}}=0\quad\text{for}\quad n\gg 0\,.
\]
On the other hand, adjunctions yield the first two isomorphisms below:
\begin{align*}
         \HH n{\Rhom Rk{\Rhom RXR}}
         &\cong \HH n{\Rhom  R{k\dtensor RX}R} \\
         &\cong \HH n{\Rhom k{k\dtensor RX}{\Rhom RkR}}\\
         &\cong {\Hom k{\hh{k\dtensor RX}}{\Ext {}RkR}}_n\\
         &= \prod_{j+i=-n}\Hom k{\HH i{k\dtensor RX}}{\Ext jRkR}\,.
\end{align*}
As $\Ext {}RkR\ne0$, it follows that one has $\HH i{k\dtensor RX}=0$
for $i\ll 0$.  At this point, one can conclude that $\HH iX=0$ holds
for $i\ll 0$ by invoking \cite[(4.5)]{FI}.  What follows is a direct
and elementary proof, which exploits an idea from \cite[(5.12)]{DGI}.

The first step is to verify that every complex $C$ with $\hh C$ of finite
length has $\HH i{C\otimes_RX}=0$ for $i\ll 0$. In the case where $C$
is a module, this is achieved by an obvious induction on its length. The
general case is settled by induction on the number of non-zero homology
modules of $C$.

Let now $C$ be the Koszul complex on a subset $\bsx=\{x_1,\dots,x_e\} $ of $R$. Recall
that $C$ is equal to $C'\otimes_RC''$, where $C'$ and $C''$ are Koszul complexes on
$\{x_1,\dots,x_{e-1}\}$ and $x_e$, respectively.  Thus, one has an exact sequence of
complexes
\[
    0\to C'\otimes_R X \to C\otimes_R X\to \Shift (C'\otimes_RX)\to 0\,.
\]
Its homology exact sequence yields exact sequences of $R$-modules
\[
       0\lra\HH i{C'\otimes_RX}/x_e\HH i{C'\otimes_RX} \lra\HH i{C\otimes_RX} \lra
       \HH{i-1}{C'\otimes_RX}
\]
By induction on $e$, one deduces that each $R$-module $\HH i{C\otimes_RX}$
is finite.

For the final step, choose $\bsx$ to generate the maximal ideal of $R$.
The length $\hh C$ then is finite, so the first step yields $\HH
i{C\otimes_RX}=0$ for $i\ll 0$.  The exact sequence above then gives
$\HH n{C'\otimes_RX}/x_e\HH n{C'\otimes_RX}=0$ for $i\ll0$, which
implies $\HH i{C'\otimes_RX}=0$ for $i\ll0$, by Nakayama's Lemma.
Splitting off one element $x_j$ at a time, we arrive at $\HH i{X}=0$
for $i\ll0$, as desired.  \end{proof}

\begin{proof}[Proof of Theorem~\emph{\ref{foxby}}]
  Set $(-)^*=\Rhom R-R$.  We start with the local case:

\begin{Claim}
  Assume that $R$ is a local ring.  If $\hh M$ is finite and there exits
  an isomorphism $\mu\col M\to M^{**}$ in $ \dcat$, then $\gdim_RM$
  is finite.
\end{Claim}

Indeed, as $\hh M$ is finite, one has $\HH i{M^*}=0$ for $i\gg 0$, and
the $R$-module $\HH i{M^*}$ is finite for each $i$.  The isomorphism
$\mu$ and Proposition~\ref{finite:test} yield $\HH i {M^*}=0$ for $i\ll
0$. Moreover, one has $\HH i{M^*}=0$ for $i\gg 0$.  It thus remains
to prove that the biduality morphism $\delta_M\col M\to M^{**}$ is an
isomorphism; see~\ref{gd=reflexive}.

The composition $(\delta_M)^*\delta_{M^*}$ is the identity map of $M^*$,
so for each $n\in\BZ$ the map $\HH n{\delta_{M^*}}\col\HH n{M^*}\to\HH
n{M^{***}}$ is a split monomorphism.  On the other hand, $\mu$ induces an
isomorphism $\HH n{M^{***}}\cong\HH n{M^*}$ for each $n\in\BZ$.  As the
$R$-module $\HH n{M^*}=\Ext{-n}RMR$ is finite, we conclude that $\HH
n{\delta_{M^*}}$ is bijective.  Thus, $\delta_{M^*}$ is an isomorphism
in $\dcat$, and hence so is $\delta_{M^{**}}$.  Since the square
\[
     \xymatrixrowsep{2.5pc} \xymatrixcolsep{3pc} \xymatrix{ M\ar@{->}[r]^{\mu}_{\simeq}
       \ar@{->}[d]_{\delta_M}
       & M^{**} \ar@{->}[d]_{\simeq}^{\delta_{M^{**}}}\\
       M^{**}\ar@{->}[r]^{\mu^{**}}_{\simeq} & M^{****}}
\]
in $\dcat$ commutes, we see that $\delta_M$ is an isomorphism, as desired.

This completes the proof of the claim.

At this point, we can conclude that when condition (a) holds,
the number $\gdim_{R_\fm}M_\fm$ is finite for each $\fm\in\Max R$.
When $\Rhom RMR$ is homologically bounded, it is homologically finite,
so one has $(\delta_M)_\fm = \delta_{M_\fm}$ for each $\fm\in\Max R$. As
$\delta_{M_\fm}$ is an isomorphism, so is $\delta_M$, that is to say,
$\gdim_RM$ is finite; see~\ref{gd=reflexive}.

It remains to prove the theorem when $\dim R$ is finite.
Since $\gdim_{R_\fm}M_\fm$ is finite for each $\fm\in \Max R$,
from \eqref{eq:gd=reflexive} we get the second equality below:
\begin{align*}
       - \inf \hh{\Rhom RMR}_\fm & = -\inf \hh{\Rhom {R_\fm}{M_\fm}{R_\fm}} \\
       & = \gdim_{R_\fm}M_\fm \\
       & = \depth R_\fm - \depth_{R_\fm}M_\fm  \\
       &\le \dim R_\fm + \sup{\hh{M_\fm}}  \\
       &\le \dim R + \sup{\hh{M}}
\end{align*}
The third one is the Auslander-Bridger Equality~\cite[(2.3.13)]{Ch},
with depth for complexes defined as in \cite[\S{A.6}] {Ch}.  The first
inequality follows from \cite[(A.6.1.1)]{Ch}, and the other inequality
is evident. Thus, $\HH i{\Rhom RMR}$ vanishes for $i\ll0$, so part (a)
applies and $\gdim_RM$ is finite.

This completes the proof of the theorem.
\end{proof}

\section{Quasi-Gorenstein homomorphisms}
\label{Quasi-Gorenstein homomorphisms}

In this section $\sigma\col K\to S$ denotes a homomorphism of rings.

One says that $\sigma$ is (\emph{essentially}) \emph{of finite type}
if it can be factored as $\sigma=\pi\varkappa$, where $\varkappa$ is the
canonical map to a (localization of a) polynomial ring in finitely many
indeterminates, and $\pi$ is a surjective homomorphism of rings.

A homomorphism of rings is \emph{local} if its source and target are local
rings, and it maps the unique maximal ideal of the source into that of the
target.  The \emph{localization} of $\sigma$ at a prime ideal $\fn$ of $S$
is the induced local homomorphism $\sigma_\fn\col K_{\fn\cap K}\to S_\fn$.

We say that $\sigma$ is \emph{Gorenstein at} some $\fn\in\Spec S$
if $\sigma_\fn$ is flat and the local ring $S_\fn/(\fn\cap K)S_\fn$
is Gorenstein; $\sigma$ is \emph{Gorenstein} when this holds for every
$\fn\in\Spec S$.

Recall that $\vf$ is called (\emph{essentially}) \emph{smooth} if it is
(essentially) of finite type, flat, and with geometrically regular fibers.
Such maps are evidently Gorenstein.

A \emph{Gorenstein{-}by{-}finite factorization} of $\sigma$ is an equality
$\sigma= \pi\varkappa$, where $\varkappa$ and $\pi$ are homomorphisms of
rings with $\varkappa$ Gorenstein and $\pi$ finite.   In a similar vein,
we speak of (essentially) smooth{-}by{-}finite factorizations, etc.
Each homomorphism (essentially) of finite type has (essentially)
smooth{-}by{-}surjective decompositions.

Finally, we recall some notions and results from \cite{AF:qG}.

\begin{chunk}
\label{factorizations}
Let $\sigma$ be a homomorphism essentially of finite type.  We say that
it has \emph{finite G-dimension} at some $\fn\in\Spec S$, and write
$\gdim\sigma_\fn<\infty$, if for some Gorenstein{-}by{-}surjective
factorization $K_{\fn\cap K}\to P'\to S_\fn$ of $\sigma_\fn$
one has $\gdim_{P'}S_\fn<\infty$.  This property does not
depend on the choice of factorization; see \cite[(4.3)]{AF:qG}.  It holds,
for instance, when $\sigma_\fn$ is flat or when $K$ is Gorenstein;
see \cite[(4.4.1), (4.4.2)]{AF:qG}.

When $\gdim_{P'}S_\fn$ is finite the complex $\Rhom
{P'}{S_\fn}{P'}$ does not depend on factorizations,
up to isomorphism and shift in $\dcat[S_\fn]$; see \cite[(6.6), (6.7),
(5.5)]{AF:qG}.

We say that $\sigma$ is \emph{quasi-Gorenstein} at $\fn$ if it has finite
G-dimension at $\fn$, and $\Rhom {P'}{S_\fn}{P'}$ is isomorphic in 
$\dcat[S_\fn]$ to some shift of $S_\fn$.  When this holds
at each $\fn\in\Spec S$ we say that it is \emph{quasi-Gorenstein}; see
\cite[(7.8)]{AF:qG}.  When $\sigma_\fn$ is flat, it is quasi-Gorenstein at
$\fn$ if and only if it is Gorenstein at $\fn$; see \cite[(8.1)]{AF:qG}.
 \end{chunk}

We give new characterizations of (quasi-)Gorenstein homomorphisms.

\begin{theorem}
\label{finite:gorcriteria}
Let $K$ be a noetherian ring, $\sigma\col K\to S$ a homomorphism of rings,
and $K\to P\to S$ a Gorenstein{-}by{-}finite factorization.

The homomorphism $\sigma$ is quasi-Gorenstein if and only if the
graded $S$-module $\Ext{}{P}S{P}$ is invertible; when it is, one has
$\gdim_PS\le\height\ann_PS$.
\end{theorem}

The notion of invertible graded  module is explained further below.

\begin{corollary}
\label{cor:gor}
When $\Spec S$ is connected the map $\sigma$ is quasi-Gorenstein
if and only if the $S$-module $\Ext{n}{P}S{P}$ is invertible for
$n=\grade_PS$, and zero otherwise.
 \end{corollary}

Set $\env S=S\otimes_KS$ and let $\mu\col \env S\to S$ denote the
multiplication map, $\mu(a\otimes b)=ab$.

\begin{theorem}
\label{diagonal:gor}
Let $K$ be a noetherian ring and let $\sigma\col K\to S$ be a flat,
essentially of finite type homomorphism of rings.

The following conditions are then equivalent.
\begin{enumerate}[\quad\rm(i)]
     \item The homomorphism $\sigma$ is Gorenstein.
     \item[\rm(i$'$)] The homomorphism $S\otimes_K\sigma\col S\to \env S$ is
       Gorenstein at each $\fm\in \Supp_{\env S}(S)$.
     \item The homomorphism $\mu\col\env S\to S$ is quasi-Gorenstein.
     \item The $\env S$-module $S$ has finite G-dimension.
     \item The graded $S$-module $\Ext{}{\env S}S{\env S}$ is invertible.
\end{enumerate}
\end{theorem}

The preceding results are proved at the end of this section, following some
technical preparation.  Condition (iv) in the second one is refined in
Theorem~\ref{gorenstein-factorization}.

Let $L$ be a finite $S$-module.  Recall that $L$ is projective if and only
if for each $\fn\in \Spec S$ the $S_\fn$-module $L_\fn$ is free. For such
a module $L$ the function $\fn \mapsto \rank_{S_\fn}L_\fn$ on $\Spec S$
is upper semi-continuous, and hence constant on each connected component.
When $\rank_{S_\fn}L_\fn=d$ for each $\fp\in\Spec S$ one says that
$L$ has \emph{rank} $d$.   Projective modules of rank $1$ are called
\emph{invertible} modules.

We say that a graded module $(E^n)_{n\in\BZ}$ is projective, respectively,
invertible, if the module $\bigoplus_{n\in\BZ}E^n$ is projective,
respectively, invertible.

\begin{lemma}
\label{graded-invertibles}
Let $S$ be a noetherian ring.  

A graded $S$-module $(E^n)_{n\in\BZ}$ is invertible if and only if
$E^n$ is finite over $S$ for each $n\in\BZ$ and $E^n_\fn\cong S_\fn$
holds for all $\fn\in\Max S\cap\Supp E^n$.  When this is the case one
has $S=\bigoplus_{i=1}^{q}J_i$ with $J_i= \ann_S(\bigoplus_{j\ne i}
E^{n_j})$, where $\{n_1,\dots,n_q\}=\{n\in\BZ\mid E^n\ne0\}$.
  \end{lemma}

\begin{proof}
The `only if' part is clear.  For the converse, only the finiteness of the $S$-module
$E=\bigoplus_{n\in\BZ}E^n$ is at issue.  Let $\fq$ be a prime ideal of $S$ and 
$\fn$ a maximal ideal containing $\fq$.  Since $S_\fq$ is indecomposable as a 
module over itself, the isomorphisms
\[
\bigoplus_{n\in\BZ}E^n_\fq\cong E_\fq\cong (E_\fn){}_{\fq S_\fn}\cong S_\fq
\]
provide a unique integer $n(\fq)$ with the property
$E_\fq=E^{n(\fn)}_\fq$.  They also imply an equality $n(\fq)=n(\fn)$, so the
function $\fq\mapsto n(\fq)$ is constant on each connected component
of $\Spec S$.  One has $E^n=0$ unless $n=n(\fq)$ holds for some $\fq$,
so $E=\bigoplus_{n\in\BZ}E^n$ has only finitely many non-zero summands.

The finite decomposition of $E$ produces a disjoint union 
\[
\Spec S=\bigsqcup_{i=1}^q\Supp_S E^{n_i}\,.
\]
It yields $S=\bigoplus_{i=1}^qJ_i$ because one has 
$\Supp_S E^{n_i}=\{\fq\in\Spec S\mid\fq\supseteq J_i\}$.
  \end{proof}

 \begin{remark}
  \label{projective_homology}
If $X$ is a complex of $S$-modules with $\HH{}X$ graded projective,
then in $\dcat[S]$ there is an isomorphism $X\simeq\HH{}X$.

Indeed, using the projectivity of $\HH{}X$, choose a section of
the canonical homomorphism of graded $S$-modules $\ZZ(X)\to\HH{}X$.
Composing this section with the inclusion $\ZZ(X)\to X$ one gets a
quasi-isomorphism of complexes $\HH{}X\to X$.
  \end{remark}

\begin{proof}[Proof of Theorem \emph{\ref{finite:gorcriteria}}] Since $P$
is a finite $P$-module, $\Ext{n}{P}S{P}$ is finite over $S$ for each
$i\in\BZ$ and for every $\fn\in\Max S$ one has isomorphisms of graded
$S$-module
  \[
      \HH{}{\Rhom {P_{\fn\cap P}}{S_\fn}{P_{\fn\cap P}}}
      \cong\Ext{}{P_{\fn\cap P}}{S_\fn}{P_{\fn\cap P}}
      \cong\Ext{}{P}S{P}_\fn
  \] 
Thus, if $\sigma$ is quasi-Gorenstein, then $\Ext{}{P}S{P}$ is
invertible by Lemma \ref{graded-invertibles}.

When $\Ext{}{P}S{P}$ is invertible one has $\Rhom PSP\simeq\Ext{}{P}S{P}$
in $\dcat[S]$, see Remark \ref{projective_homology}.  For every
$\fn\in\Spec S$ this yields an integer $d(\fn)$ and an isomorphism $\Rhom
{P_{\fn\cap P}}{S_\fn}{P_{\fn\cap P}}\simeq\Shift^{d(\fn)}S_\fn$
in $\dcat[S_\fn]$.  Corollary~\ref{invertible} implies that
$\gdim_{P_{\fn\cap P}}S_\fn$ is finite, so $P\to S$ is quasi-Gorenstein
at $\fn$; see \ref{factorizations}. The same corollary also yields
$\gdim_{P}S\leq \height \ann_PS$.  As $K\to P$ is Gorenstein, it is also
quasi-Gorenstein, hence so is $\sigma$; see \cite[(8.9)]{AF:qG}.
\end{proof}

\begin{proof}[Proof of Corollary \emph{\ref{cor:gor}}]
One has $\Ext gPSP\ne 0$ for $g=\grade_{P}S$, by definition. Thus, when $\Spec S$ is connected, the graded module $S$-module $\Ext{}{P}S{P}$ is invertible if and only if $\Ext{n}{P}S{P}$ is invertible for $n=g$ and zero otherwise; see Lemma \ref{graded-invertibles}. The desired result now follows from Theorem~\ref{finite:gorcriteria}.
 \end{proof}

For $\fn\in\Spec S$, we write $k(\fn)$ for the field of fractions
of $S/\fn$.

\begin{proof}[Proof of Theorem \emph{\ref{diagonal:gor}}]
Set $Q=\env S$ and $\psi=S\otimes_K\sigma\col S\to Q$.  This map is
flat along with $\sigma$, and hence so is $\psi_{\fm}$ for each
$\fm\in\Spec Q$; in particular, $\gdim\psi_{\fm}$ is finite.

(i)$\implies$(i$'$).  
Setting $\fn=\fm\cap S$, note the isomorphism of rings
  \[
        k(\fn)\otimes_SQ\cong k(\fn)\otimes_{k(\fn\cap K)}(k(\fn\cap
        K)\otimes_KS)\,.
  \] 
It shows that if the ring $k(\fn\cap K)\otimes_KS$ is Gorenstein, then so is 
$k(\fn)\otimes_SQ$, hence also $(k(\fn)\otimes_SQ)_{\fn}$,
which is isomorphic to $k(\fn)\otimes_{S_\fn}Q_\fm$.

(i$'$)$\implies$(i).  Every prime ideal $\fn\in\Spec S$ is the contraction of the
prime ideal $\fm=\mu^{-1}(\fn)$ of $\env S$, which contains $\Ker(\mu)$.
Thus, when (i$'$) holds the local homomorphism $\psi_{\fm}\col S_\fn\to Q_\fm$ 
is flat with Gorenstein closed fiber.  The \emph{proof} of \cite[(6.6)]{AF:Gor}
shows that then so is the local homomorphism $\sigma_\fn\col K_{\fn\cap
K}\to S_\fn$.

(iv)$\iff$(ii)$\implies$(iii).  Apply Theorem~\ref{finite:gorcriteria}
to $K=P=\env S$ and $\varkappa=\id^{\env S}$

(iii)$\implies$(i$'$)$\implies$(ii).  These assertions follow from the
Decomposition Theorem \cite[(8.10)]{AF:Gor}, applied to the evidently
quasi-Gorenstein composition $\mu\psi=\id^S$.
 \end{proof}

\section{Bigrade of a homomorphism}
\label{nBigrade of a homomorphism}

Invariants provided by Hochschild cohomology reflect the structure
of an algebra $\sigma\col K\to S$ as a \emph{bimodule} over itself.
We define the \emph{bigrade} of $\sigma$ by the formula
\[
    \bigrade{\sigma}{}=\inf\{n\in\BZ\mid\Ext n{\env S}S{\env S}\ne0\}\,.
\]

Applications of this invariant are given in the following sections.
Here we examine its formal properties and compare it to other invariants
of the $K$-algebra $S$.  For each $\fn\in\Spec S$ we define the
 \emph{residual transcendence degree} of $\sigma$ at $\fn$ as the
 number
\[
    \rtd{\sigma}{\fn}=\trdeg_{k(\fn\cap K)}k (\fn)\,.
\]

When $\varkappa\col K\to P$ is an essentially smooth homomorphism of
commutative rings the $P$-module of K\"ahler differentials $\Omega_{P|K}$
is finite and projective; see \cite[\S16.10]{Gr}.  Following \cite{YZ}, we say 
that $\varkappa$ has \textit{relative dimension} $d$ if this projective
module has rank $d$; see Section \ref{Quasi-Gorenstein homomorphisms}.
An example is given by the canonical map $K\to U^{-1}K[x_1,\dots,x_d]$,
where $x_1,\dots,x_d$ are indeterminates and $U$ is any multiplicatively
closed set.  Thus, every homomorphism (essentially) of finite type has
an (essentially) smooth{-}by{-}surjective factorization of finite relative
dimension.

  \begin{theorem}
\label{thm:bigrade}
Let $K$ be a noetherian ring, $\sigma\col K\to S$ a flat homomorphism,
and $K\xra{\varkappa} P\to S$ an essentially smooth{-}by{-}finite
factorization of relative dimension $d$.
  \begin{enumerate}[\rm(1)]
 \item
For every minimal prime ideal $\fq$ of $S$ there are (in)equalities
  \[
0\le d-\pd_PS\le\bigrade\sigma\le\bigrade[\fq]\sigma=\rtd{\sigma}{\fq}\le
d \,.
  \]
    \item
The following conditions are equivalent
    \begin{enumerate}[\quad\rm(i)]
  \item
$\bigrade\sigma=d-\pd_PS$.
  \item
$\Ass S\cap\Supp_S\Ext{p}PSP\ne\varnothing$ for $p=\pd_PS$.
    \end{enumerate} \item
When $S$ and $P$ are integral domains one has
  \[
\trdeg_{\sigma}(0)=d-\grade_PS\,.
  \]
    \end{enumerate}
  \end{theorem}

The hypotheses and notation of the theorem stay in force for the
rest of this section.  Its proof is presented following that of Lemma
\ref{lem:bigrade3}.

The first inequality in (1) is a consequence of a general result that
tracks homological dimensions along smooth homomorphisms:

\begin{chunk}
\label{fpd:smooth}
For every complex $M$ of $P$-module with $\hh M$ finite one has inequalities
\[
      \fd_KM\le\pd_PM\le\fd_KM+ d\,.
\]
See \cite[1.4]{AIL} for a proof.  In particular, $\pd_PM$ and $\fd_KM$ are
finite simultaneously.  \end{chunk}

The next result, proved in \cite[4.2]{AIL}, is a critical ingredient in
the arguments below.

\begin{chunk}
\label{computation:smooth}
For every $n\in\BZ$ one has an isomorphism of $S$-modules
\[
    \Ext n{\env S}S{\env S}\cong 
    \Ext {n-d}S{\Rhom PS{{\ts\bwedge}^{d}_P\Omega_{P\var K}}}S\,.
\]
\end{chunk}

The next remark is useful in applications of the reduction formula above.

\begin{chunk}
\label{supports}
The following subsets of $\Spec S$ are equal:
  \begin{equation}
  \label{eq:supports2}
      \Supp_S\Ext{n}PS{\bwedge^d_P\Omega_{P|k}}=\Supp_S\Ext nPSP\,.
  \end{equation}

Indeed, set $V=\bwedge^d_P\Omega_{P|k}$.  Since the $P$-module $V$
is invertible, for each $n\in\BZ$ and every $\fn\in\Spec S$ there are
isomorphisms of $S_\fn$-modules
  \begin{equation}
  \label{eq:supports1}
\begin{aligned}
        \Ext nPS{\bwedge^d_P\Omega_{P|k}}_\fn
        &\cong \Ext nPS{\bwedge^d_P\Omega_{P|k}}\otimes_S S_\fn
        \\
        &\cong \bwedge^d_P\Omega_{P|k}\otimes_P\Ext nPSP\otimes_S S_\fn
        \\
        &\cong \bwedge^d_P\Omega_{P|k}\otimes_P(\Ext nPSP_\fn)
        \\
        &\cong(\bwedge^d_P{\Omega_{P|k}})_{\fn\cap P}\otimes_{P_{\fn\cap P}}\Ext nPSP_\fn
        \\
        &\cong \Ext nPSP_\fn
\end{aligned}
  \end{equation}
\end{chunk}

\begin{lemma}
\label{lem:bigrade1}
There is an inequality $\bigrade\sigma\ge d-\pd_PS$.

Equality holds if and only if $\Ass S\cap\Supp_S\Ext{p}PSP\ne \varnothing$.
\end{lemma}

\begin{proof}
Set $D=\Rhom PS{\bwedge^d_P\Omega_{P|k}}$ and $C=\HH{-p}D$, where $p=\pd_PS$.

{}From~\eqref{eq:supports2} one gets $\HH nD=0$ for $n<-p$.  This implies the
second isomorphism of $S$-modules below, while \ref{computation:smooth}
gives the first one:
 \[
  \Ext{n+d}{\env S}{S}{\env S}\cong \Ext{n}SDS\cong
    \begin{cases}
      0&\text{for }n<-p\,;
    \\
      \Hom SCS\quad&\text{for }n=-p\,.
    \end{cases}
 \]
These isomorphisms yield $\bigrade\sigma\ge d-p$ and show that equality
is equivalent to $\Hom SCS\ne0$.  Referring to a standard formula and
to~\eqref{eq:supports2} we obtain
 \[
  \Ass_S\Hom SCS=\Ass S\cap\Supp_SC=\Ass S\cap\Supp_S\Ext{p}PSP\,.
 \]
Thus, $\Hom SCS\ne0$ is equivalent to $\Ass
S\cap\Supp_S\Ext{p}PSP\ne\varnothing$.  \end{proof}

Before continuing, we recall another canonical isomorphism.

\begin{chunk}
\label{tools}
Fix a prime ideal $\fn$ of $S$ and set
$\env{(S_\fn)}=S_\fn\otimes_{K_{\fn\cap K}}S_\fn$.

For each $n\in\BZ$ there is an isomorphism of $S_\fn$-modules
\[
  \Ext n{\env{(S_\fn)}}{S_\fn}{\env{(S_\fn)}}\cong\Ext n{\env S}{S}{\env S}_ \fn\,.
\]

Indeed, let $\lambda\col S\to S_\fn$ and $\kappa\col K\to
K_{\fn\cap K}$ denote the localization maps.  The homomorphism
of rings $\lambda\otimes_\kappa\lambda\col\env S\to\env{(S_\fn)}$ is flat, and
one has an isomorphism $S_\fn\cong\env{(S_\fn)}\otimes_{\env S}S$ of
$\env{(S_\fn)}$-modules, whence the first isomorphism below:
\begin{align*}
    \Ext n{\env{(S_\fn)}}{S_\fn}{\env{(S_\fn)}}
    &\cong\Ext n{\env S}{S}{\env S}\otimes_{\env S}{\env{(S_\fn)}}\\
    &\cong\Ext n{\env S}{S}{\env S}\otimes_SS_\fn\,.
\end{align*}
For the second one note that $\env S$ acts on $\Ext n{\env S}{S}{\env S}$
through $S$.  \end{chunk}

\begin{lemma}
\label{lem:bigrade2}
The following equality holds:
\begin{equation*}
    \bigrade{\sigma}{}=\inf\{\bigrade[\fn]{\sigma}\mid \fn\in\Spec S\}\,.
  \end{equation*}
\end{lemma}

\begin{proof}
  Set $g=\bigrade{\sigma}$.  {}From the isomorphisms in~\ref{tools} one reads off an
  inequality $\bigrade[\fn]{\sigma}\ge g$, which becomes an equality when $\fn$ is in
  $\Supp_S \Ext g{\env S}{S}{\env S}$.
\end{proof}

\begin{lemma}
 \label{smooth:reldim}
For each prime ideal $\fm$ in $P$ one has $\trdeg_{\varkappa}k(\fm)\leq d$.

Equality holds when $P$ is a domain and $\fm=(0)$.
\end{lemma}

  \begin{proof}
Set $k=k(\fm\cap K)$ and $P'=(k\otimes_KP)_\fm$.

The composed homomorphism $\varkappa'\col k\to k\otimes_KP \to
P'$ is essentially of finite type.  Its fibers are among those of
$\varkappa$, so it is essentially smooth, and the canonical isomorphism
$\Omega_{P'|k}\cong(k\otimes_K\Omega_{P|K})_\fm$ shows that it has
relative dimension $d$.

The surjection $P'\to k(\fm)$ induces a surjection $\omega\col
k(\fm)\otimes_{P'}\Omega_{P'|k}\to\Omega_{k(\fm)|k}$.  This gives the
second inequality below, and \cite[(26.10)]{Ma} provides the first one:
  \[
     \trdeg_kk(\fm) \le \rank_{k(\fm)}\Omega_{k(\fm)|k} \le d\,.
  \]

When $P$ is a domain and $\fm=(0)$ one has $P'=k(\fm)$.  In particular,
$\omega$ is an isomorphism, and thus the second inequality above becomes
an equality.  The first inequality also does, as the homomorphism $k\to
k(\fm)$ is essentially smooth.
 \end{proof}

Let $\Min S$ denote the set of minimal prime ideals of $\Spec S$.

\begin{lemma}
  \label{lem:bigrade3}
For all prime ideals $\fq\in\Min S$ and $\fn\in\Spec S$ with $\fq \subseteq\fn$ one has
     \[
     \bigrade{\sigma}\le\bigrade[\fn]{\sigma}\le\bigrade[\fq]{\sigma}=\rtd
     {\sigma}{\fq}\le d\,.
     \]
\end{lemma}

\begin{proof}
Both inequalities on the left come from Lemma~\ref{lem:bigrade2}, as
one has $\sigma_\fq=(\sigma_\fn)_{\fq S_\fn}$.

Set $\fp=\fq\cap K$.  The rings $S_\fq$ and $K_\fp$ are artinian,
the first one because the ideal $\fq$ is minimal, the second because
$\sigma_\fq\col K_\fp\to S_\fq$ is a flat local homomorphism.

Set $k=k(\fp)$, $l=k(\fq)$, and $t=\trdeg_kl$.  Choose in $S_\fq$
elements $y_1,\dots,y_t$ that map to a transcendence basis of $l$
over $k$.  Let $x_1,\dots,x_t$ be indeterminates over $K_\fp$ and $Q$
the localization of $K_\fp[x_1,\dots,x_t]$ at the prime ideal $\fp
K_\fp[x_1,\dots,x_t]$; this is a local ring with maximal ideal $\fp Q$
and residue field $k'=k(x_1,\dots,x_t)$.

The homomorphism of $K_\fp$-algebras $K_\fp[x_1,\dots,x_t]\to S_\fq$
sending $x_i$ to $y_i$ for $i=1,\dots,t$ induces a local homomorphism
$\vf\col Q\to S_\fq$.  A length count yields
\[
      \length_{Q}(S_\fq)=\length_{S_\fq}(S_\fq)\length_{k'}(l)<\infty\,,
\]
so $\vf$ is finite.  Let $\kappa$ denote the composition $K_\fp\to
K_\fp[x_1,\dots,x_t]\to Q$.  It is local, flat, essentially of finite
type, and the fiber $Q\otimes_{K_\fp}k$ is equal to $k'$.

One has $\sigma_\fq=\vf\kappa$, so this is a smooth{-}by{-}finite
factorization of relative dimension $t$ by the discussion above.
The finite $Q$-module $S_\fq$ has finite projective dimension by
Lemma~\ref{fpd:smooth}, so it is free because $Q$ is artinian.  By the
same token, one has $\Supp_{S_\fq}\Ext 0Q{S_\fp}Q=\fq S_\fq=\Ass{S_
\fq}$, so Lemma~\ref{lem:bigrade1} yields $\bigrade[\fq]{\sigma}=t$.

Finally, set $\fm=\fq\cap Q$.  As the field extension $k(\fm) \subseteq
k(\fq)$ is finite, one has $t=\trdeg_k{k(\fm)}$.  On the other hand,
Lemma \ref{smooth:reldim} yields $t\le d$.
 \end{proof}

  \begin{proof}[Proof of Theorem \emph{\ref{thm:bigrade}}]
(1)  The first inequality comes from \ref{fpd:smooth}, the second
one from Lemma~\ref{lem:bigrade1}, the remaining relations from
Lemma~\ref{lem:bigrade3}.

(2)  The desired assertion is part of Lemma~\ref{lem:bigrade1}.

(3)  Assume that $S$ and $P$ are integral domains and set $\fm=\Ker(P
\to S)$.

For each $\fn\in \Spec S$, the projective dimension of $S_{\fn}$ over
$P_{\fn\cap P}$ is finite, see \ref{fpd:smooth}.  Thus \cite[(2.5)]{AF:cm}
yields that $ \grade_{P_{\fn\cap P}}S_{\fn}= \dim P_{\fm}$, so one obtains
\[
      \grade_PS=\inf\{\grade_{P_{\fn\cap P}}S_{\fn}\var\fn\in\Spec S\}= \dim P_{\fm}\,.
\]
Set $S'=S_{(0)}$ and $K'=K_{(0S)\cap K}$.  As $S'$ is the residue field
of $P_\fm$, and $K'\to S'$ is a flat local homomorphism, one sees that
$K'$ is a field.  The local domain $P_\fm$ is the localization of some
finitely generated $K'$-algebra $P'$ at a prime ideal $\fm'$, so one has
 \[
      \dim P_{\fm}=\height(\fm')=\dim
      P'-\dim(P'/\fm')=\trdeg_{K'}P'_{(0S)\cap P'}-\trdeg_{K'}S'\,.
 \]

To finish the proof, note that Lemma~\ref{smooth:reldim} yields
$d=\trdeg_{K'}P'$.
  \end{proof}

Formal properties of Hochschild cohomology have implications
for bigrade:

\begin{remark}
\label{bigrade:base_change}
For any homomorphism of rings $K\to K'$ we identify $K'$ and
$K\otimes_KK'$ via the canonical isomorphism, set $S'=S\otimes_KK'$, and
note that $\sigma\otimes_KK'\col K'\to S'$ is (essentially) of finite type
and flat, along with $\sigma$.  Also, set $\env{S'}=S'\otimes_{K'}S'$.

When $K\to K'$ is flat so is $\env S\to\env{S'}$, due to the canonical
isomorphism of $K'$-algebras $\env{S'}\cong\env S\otimes_KK'$.  Thus,
for each $n\in\BZ$ one gets isomorphisms
\[
      \Ext n{\env{S'}}{S'}{\env{S'}} \cong \Ext n{\env S}S{\env{S'}} \cong \Ext n{\env
        S}{S}{\env S}\otimes_SS'\,.
\]

As a consequence, for every flat homomorphism $K\to K'$ one obtains
\[
      \bigrade{\sigma}{}\le\bigrade{\sigma\otimes_KK'}\,;
\]
equality holds when $K'$ is faithfully flat over $K$.
\end{remark}

\begin{remark}
  \label{tools2}
Let $\tau\col K\to T$ be a flat homomorphism essentially of finite type.  

For $R=S\times T$, and for each $n\in\BZ$ there are canonical isomorphisms
   \begin{gather*}
\Ext{n}{\env R}{R}{\env R}
  \cong \big(S\otimes_R\Ext{n}{\env R}{R}{\env R}\big)
  \oplus \big(T\otimes_R\Ext{n}{\env R}{R}{\env R}\big)\,,
    \\
S\otimes_R\Ext{n}{\env R}{R}{\env R}
  \cong\Ext{n}{\env S}{S}{\env S}\,,
    \\
T\otimes_R\Ext{n}{\env R}{R}{\env R}
  \cong\Ext{n}{\env T}{T}{\env T}\,,
   \end{gather*}
of $R$-modules, $S$-modules, and $T$-modules, respectively.

In particular, for the diagonal map $\delta\col K\to K\times K$ the
following equality holds:
\[
  \bigrade{(\sigma\times\tau)\delta}=\min\{\bigrade\sigma\,,\,\bigrade \tau\}\,.
\]
\end{remark}

We illustrate Theorem \ref{thm:bigrade} in a concrete situation.  When every
associated prime ideal $\fq$ of $S$ satisfies $\dim(S/\fq)=\dim S$
we say that $S$ is \emph{equidimensional}.

\begin{proposition}
\label{graded}
Let $K$ be a field and $S$ an $\BN$-graded $K$-algebra, generated by
finitely many elements in $S_{\ges 1}$.  For the inclusion $\sigma\col
K \to S$ one then has
\[
      \depth S\le\bigrade\sigma\le\dim S\,.
\]
The first inequality is strict when $S$ is equidimensional, but not
{Cohen-Macaulay}.
 \end{proposition}

 \begin{proof}
Set $d=\dim S$, and let $P$ be the $K$-subalgebra of $S$ generated by
some homogeneous system of parameters.  Thus, $P$ is a polynomial ring
in $d$ variables and $K\to P\to S$ is a smooth{-}by{-}finite factorization
of relative dimension $d$.

Set $p=\pd_PS$.  The Auslander-Buchsbaum Equality gives $d-p=\depth S$, so
Theorem~\ref{thm:bigrade}(1) gives the desired inequalities.  When $S$ is
equidimensional and $\fq$ is in $\Ass S$ one has $\dim(S/\fq)=d= \dim P$.
As $S/\fq$ is finite over $P$ and $P$ is a domain, this implies $\fq\cap
P=(0)$. When $S$ is not {Cohen-Macaulay} one has $p>0$, hence
  \begin{align*}
        \Ext pPSP_{\fq} &\cong(\Ext p{P}{S}{P}\otimes_PP_{(0)})\otimes_{P_{(0)}}S_\fq
        \\
        &\cong\Ext p{P_{(0)}}{S_{(0)}}{P_{(0)}}\otimes_{P_{(0)}}S_\fq
        \\
        &=0
  \end{align*} 
{}From Theorem~\ref{thm:bigrade}(2) we now conclude that
$\bigrade\sigma>d-p$ holds.
 \end{proof}

Next we show that the second inequality in the proposition can be strict
as well.

\begin{example}
  When $K$ is a field of characteristic zero the subring
\[
      S=K[x^3,x^2y,x^2z,xy^2,xz^2,y^3,y^2z,yz^2,z^3]
\]
of a polynomial ring $K[x,y,z]$ and the inclusion map $\sigma\col K\to S$ satisfy
\[
      \bigrade{\sigma}=2<3=\dim S\,.
\]

Indeed, for the equality on the right note that the field of fractions
of $S$ is equal to $K(x,y,z)$.  The one on the left results from the
isomorphisms
\[
      \Ext n{\env S}S{\env S} \cong
      \begin{cases}
        0\quad&\text{for }n\le1\,;
        \\
        K\quad&\text{for }n=2\,.
      \end{cases}
\]
For $K=\BQ$ they are established through a computation with
\textsl{Macaulay~2}.  The general case follows from here and
Remark~\ref{bigrade:base_change}.
 \end{example}

\section{{Cohen-Macaulay} homomorphisms}
\label{{Cohen-Macaulay} homomorphisms}

Recall that a flat homomorphism $K\to S$ is said to be
\emph{{Cohen-Macaulay} at a prime ideal} $\fn$ of $S$ if the local
ring $S_\fn/(\fn\cap K)S_\fn$ is {Cohen-Macaulay}; the homomorphism is
\emph{{Cohen-Macaulay}} if it has this property at each $\fn\in \Spec S$.

We fix notation and hypotheses for the rest of the section:

 \begin{chunk}
\label{setup:cm}
Let $K$ be a noetherian ring, $\sigma\col K\to S$ a flat homomorphism
essentially of finite type, and $K\to P\to S$ an essentially
smooth{-}by{-}surjective factorization of $\sigma$ of relative dimension
$d$.  We assume that $\Spec S$ is connected (but see Remark~\ref{tools2}).
  \end{chunk}

\begin{theorem}
\label{defs:cm2}
The following conditions are equivalent.
\begin{enumerate}[\rm\quad(i)]
      \item $\sigma$ is {Cohen-Macaulay}.
      \item $\grade_PS=\grade_{P_{\fn\cap P}}S_\fn=\pd_{P_{\fn\cap P}}S_\fn=\pd_PS$
for every $\fn\in\Spec S$.
      \item $\grade_PS=\pd_PS$.
     \item $\Ext nPSP=0$ for $\grade_PS< n\le d$.
\end{enumerate}
The homomorphism $\sigma$ is Gorenstein if and only if it is Cohen-Macaulay and the $S$-module 
$\Ext gPSP$ is invertible.  
  \end{theorem}

The theorem is proved after some reminders  on
localizations of homomorphisms.

\begin{chunk}
\label{imperfection}
For each $\fn\in\Spec S$ the following inequalities hold:
\[
\grade_PS\leq\grade_{P_{\fn\cap P}}S_\fn \leq\pd_{P_{\fn\cap P}}S_\fn\leq\pd_PS\le d\,.
\]
The first three are standard, the last one comes from Lemma~\ref {fpd:smooth}.
\end{chunk}

\begin{chunk}
\label{cm:localization}
The homomorphism $\sigma$ is {Cohen-Macaulay} if and only if for every
$\fn\in\Spec S$ one has $\pd_{P_{\fn\cap P}}S_\fn=\grade_{P_{\fn\cap
P}}S_\fn$; this follows from \cite[(3.5) and (3.7)]{AF:cm}.
  \end{chunk}

\begin{proof}[Proof of Theorem \emph{\ref{defs:cm2}}]
The flat homomorphism $\sigma$ is Gorenstein if and only
if it is quasi-Gorenstein, see \cite[(8.1)]{AF:qG}, so the desired 
criterion for Gorensteiness is a special case of Corollary
\ref{cor:gor}.  The rest of the proof is devoted to 
establishing the equivalence of the conditions in 
Theorem \ref{defs:cm2}.  Set $g=\grade_PS$ and $p=\pd_PS$.

(i)$\implies$(ii).  We start by proving that for all $\fn,\fn'\in\Spec
S$ one has
\[
      \grade_{P_{\fn\cap P}}S_\fn=\pd_{P_{\fn\cap P}}S_\fn= \grade_{P_{\fn'\cap
          P}}S_{\fn'}=\pd_{P_{\fn'\cap P}}S_{\fn'}\,.
\]

For the first and last equalities it suffices to remark that
$\sigma_{\fn}$ and $\sigma_{\fn'}$ are {Cohen-Macaulay} along with
$\sigma$, then refer to~\ref{cm:localization}.  When $\fn$ is contained
in $\fn'$ the equality in the middle is obtained by applying the chain of
inequalities in~\ref{imperfection} to $\sigma_{\fn'}$.  Since $\Spec S$
is connected, when $\fn$ and $\fn'$ are arbitrary one can find in $\Spec
S$ a path
\[
      \fn=\fn_0\supseteq\fn_1\subseteq\fn_2\subseteq\cdots
      \supseteq\fn_{j-1}\subseteq\fn_j=\fn'
\]
The already treated case of embedded prime ideals shows that the
invariants that we are tracking stay constant on each segment of such
a path.

When $\fn$ ranges over $\Spec S$ the infimum of $\grade_{P_{\fn\cap
P}}S_\fn$ equals $g$, and the supremum of $\pd_{P_{\fn\cap P}}S_\fn$
equals $p$, so one gets $g=p$.

(ii)$\implies$(iii)$\implies$(iv) These implications are evident.

(iv)$\implies$(iii).  
{}From \ref{fpd:smooth} one gets $p\le d$.  Since $P$ is noetherian and
$S$ is a finite $P$-module, one has $\Ext pPSP\ne0$.  The definition
of grade and the hypothesis imply that for $n\in[0,d]$ one has $\Ext
nPSP\ne0$ only when $n=g$, so $p=g$ holds.

(iii)$\implies$(ii).  This follows from \ref{imperfection}.

(ii)$\implies$(i).  This follows from \ref{cm:localization}.
  \end{proof}

In the next result we collect some properties of the $S$-module
$\Ext{}PS{{\ts\bwedge}^{d}_P\Omega_{P\var K}}$, for use in the 
next section.  The set up is as in \ref{setup:cm}.

\begin{theorem}
\label{big:cm}
Assume that $\sigma$ is {Cohen-Macaulay} and set $p=\pd_PS$.

For all $\fq\in\Min S$ and $\fn\in\Spec S$ there are equalities
\begin{equation}
  \label{eq:cm}
      \bigrade{\sigma}=\bigrade{\sigma_\fn}=\trdeg_{\sigma}k(\fq)\,.
\end{equation}

The $S$-module $C=\Ext pPS{{\ts\bwedge}^{d}_P\Omega_{P\var K}}$, has
the following properties.
\begin{enumerate}[\rm(1)]
\item
There is an equality $\Ass_SC=\Ass S$.
\item
For each $n\in\BZ$ there is an isomorphism of $S$-modules
  \[
  \Ext n{\env S}S{\env S}\cong \Ext {n-b}SCS \quad\text{where}\quad b=\bigrade\sigma\,.
  \]
\item
If $K$ is Gorenstein, then $C_\fn$ is a canonical module for $S_\fn$ for each
  $\fn\in\Spec S$.
\end{enumerate}
\end{theorem}

\begin{proof}
Set $V={\ts\bwedge}^{d}_P\Omega_{P\var K}$.  We deal with $C=\Ext pPSV$
first.

Let $F\xra{\simeq}S$ be a resolution with each $F_i$ finite projective
over $P$ and $F_i=0$ for $i\notin[0,p]$.  Theorem \ref{defs:cm2} yields
$\grade_PS=p$, so one has $\Ext nPSV=0$ for $n\ne p$; see~\eqref{eq:supports2}.
Thus, one gets quasi-isomorphisms of complexes of $P$-modules
\begin{equation}
  \label{eq:dualizing}
  \Rhom PSV\simeq\Hom PFV\simeq\Shift^{-p}C
\end{equation}
Each $\Hom P{F_i}V$ is finite projective, and one has isomorphisms of
complexes
\begin{align*}
  \Hom P{\Hom PFV}V &\cong\Hom P{{\Hom PFP}\otimes_PV}V
  \\
  &\cong\Hom P{\Hom PFP}{\Hom PVV}
  \\
  &\cong\Hom P{\Hom PFP}P
  \\
  &\cong F
\end{align*}
These computations localize.  In particular, for each $\fn\in\Spec S$ one gets
\begin{equation}
  \label{eq:pd}
\pd_{P_{\fn\cap P}}C_{\fn}=\pd_{P_{\fn\cap P}}S_{\fn}\,.
\end{equation}

(1)  An ideal $\fn\in\Spec S$ is associated to $C$ if and only if
$\depth_{S_\fn}{C_\fn}=0$.  The finiteness
of $S_\fn$ as a $P_{\fn\cap P}$-module and the Auslander-Buchsbaum
Equality yield
\[
\depth_{S_\fn}C_\fn=\depth_{P_{\fn\cap P}}C_\fn= \depth_{P_{\fn\cap P}}P_{\fn\cap
  P}-\pd_{P_{\fn\cap P}}C_{\fn}\,.
\]
Thus, $\fn\in\Ass_SC$ is equivalent to $\pd_{P_{\fn\cap P}}C_{\fn}=
\depth_{P_{\fn\cap P}}P_{\fn\cap P}$.  Similarly, $\fn\in\Ass S$ amounts 
to $\pd_{P_{\fn\cap P}}S_{\fn}=\depth_{P_{\fn\cap P}}P_{\fn\cap P}$.  
Now \eqref{eq:pd} gives $\Ass_SC=\Ass S.$

(2) From \eqref{eq:dualizing} and \ref{computation:smooth} for each $n\in\BZ$ one
gets
\[
\Ext n{\env S}S{\env S}\cong\Ext {n-d+p}SCS\,.
\]
We have just proved $\Ass S\cap\Supp_S\Ext pPSP\ne\varnothing$, so
Theorem~\ref{thm:bigrade}(2) implies
\begin{equation}
  \label{eq:defect}
  \bigrade\sigma=d-p\,.
\end{equation}

Now we can prove \eqref{eq:cm}.  Pick $\fn$ in $\Spec S$.
The induced homomorphisms $K_{\fn\cap K}\to P_{\fn\cap P}\to S_\fn$
provide a smooth{-}by{-}surjective factorization of relative dimension $d$
of $\sigma_\fn\col K_{\fn\cap K}\to S_\fn$.  We have already proved
that $\sigma_\fn$ is {Cohen-Macaulay}, and $\pd_{P_{\fn\cap P}}S_\fn=p$
holds, so we get $\bigrade\sigma=\bigrade[\fn]\sigma$ from formula
\eqref{eq:defect} applied to $\sigma_\fn$.  This proves the first equality.
Theorem~\ref{thm:bigrade} gives the second one.

(3) When  $K$ is Gorenstein the rings $P$ and $S$ are Gorenstein
and {Cohen-Macaulay}, respectively, as they are flat over $K$ with
fibers of the corresponding type.  For $\fn\in\Spec S$ and $\fm=\fn\cap P$
Theorem~\ref{defs:cm2}(1) gives $\pd_{P_\fm}S_\fn=p$ and 
\eqref{eq:supports1} gives $C_\fn\cong\Ext p{P_{\fm}}{S_\fn}{P_{\fm}}$, so
$C_\fn$ is a canonical module for $S_\fn$ by \cite[(3.3.7)]{BH}, .
 \end{proof}

\section{Gorenstein homomorphisms}
\label{Gorenstein homomorphisms}

Combining earlier results we get a `structure theorem' for Gorenstein 
algebras.

\begin{theorem}
\label{gorenstein-factorization}
If $K$ is a noetherian ring and $\sigma\col K\to S$ is a Gorenstein 
homomorphism essentially of finite type, then for some $q\ge0$ one has
  \[
\{n\in\BZ\mid\Ext n{\env S}S{\env S}\ne0\}=\{n_1,\cdots,n_q\}\,.
  \]
Furthermore, there is an isomorphism of $K$-algebras
  \[
S\cong\prod_{i=1}^qS_i
 \quad\text{with}\quad
S_i=S/\ann\big(\Ext{n_i}{\env S}S{\env S}\big)\,,
  \]
for each $i$ the map $K\to S\to S_i$ is Gorenstein, $n_i=\rtd{\sigma}{\fq}$ for every
$\fq$ in $\Min S_i$, the $S_i$-module $\Ext{n_i}{\envv Si}{S_i}{\envv Si}$ is 
invertible, and $\Ext{n}{\envv Si}{S_i}{\envv Si}=0$ for $n\ne n_i$.
 \end{theorem}

\begin{proof}
The graded $S$-module $\Ext{}{\env S}S{\env S}$ is invertible by
Theorem~\ref{diagonal:gor}. In particular, $\Ext n{\env S}S {\env S}$
is not zero for finitely many values of $n$, say $n_1,\dots,n_q$.
Lemma~\ref{graded-invertibles} provides the decomposition of $S$ and
Remark \ref{tools2} an equivariant isomorphism
  \[
\Ext{}{\env S}S{\env S}\cong \bigoplus_{i=1}^q\Ext{}{\envv Si}{S_i}{\envv
Si}
  \]
of graded modules.  It implies that the $S_i$-module $\Ext{n}{\envv
Si}{S_i}{\envv Si}$ is invertible for each $n\in\{n_1,\cdots,n_q\}$
and is zero otherwise.  Theorem~\ref{diagonal:gor} now shows that
every composition $K\to S\to S_i$ is Gorenstein.  In particular, it is
{Cohen-Macaulay}, so Theorem \ref{big:cm} yields $n_i=\rtd{\sigma}{\fq}$
for every $\fq\in\Min S_i$.
  \end{proof}

Next we search for a converse, in the spirit  of the conjecture
in the introduction.

One says that $S$ is \emph{generically Gorenstein} if for each $\fq\in
\Min S$ the ring $S_\fq$ is Gorenstein.  The next result was proved
independently in \cite[(2.1)]{ABS} and \cite[(2.2)]{HH}:

  \begin{chunk}
    \label{tachikawa}
If $S$ is a generically Gorenstein, {Cohen-Macaulay} local ring with
canonical module $C$, and $\Ext n{S}{C}{S}=0$ holds for $1\le n \le \dim
S$, then $S$ is Gorenstein.
  \end{chunk}

When $K$ is a Gorenstein ring, a homomorphism $\sigma\col K\to  S$ with
$S$ noetherian is Gorenstein if and only if the ring $S$ is Gorenstein;
see \cite[(23.4)]{Ma}.  Thus, the result below is a reformulation of
Theorem~\ref{iGorenstein} from the introduction.

\begin{theorem}
\label{portmanteau:Gor}
Let $K$ be a Gorenstein ring, $S$ a {Cohen-Macaulay} ring with connected
spectrum and $K\to S$ a flat homomorphism essentially of finite type.

If $S$ is generically Gorenstein, and for some minimal prime ideal $\fq$
of $S$ one has \[
      \Ext{n}{\env S}S{\env S}=0\quad\text{for}\quad \rtd\sigma\fq<n\le
      \rtd\sigma\fq+\dim S\,,
\] then the ring $S$ is Gorenstein.  \end{theorem}

\begin{proof}
Note that $\sigma$ is {Cohen-Macaulay} because it is flat and its
target is a {Cohen-Macaulay} ring; see \cite[(23.3), Cor.]{Ma}.
Theorem~\ref{big:cm}(3) then yields a finite $S$-module $C$, such
that $C_\fn$ is a canonical module for the {Cohen-Macaulay} local
ring $S_\fn$ for each $\fn\in\Spec S$.  Theorem~\ref{big:cm}(2) and
formula~\eqref{eq:cm} translate our hypothesis into equalities $\Ext
n{S_\fn}{C_\fn}{S_\fn}=0$ for $1\le n \le \dim S$.  It remains to invoke
\ref{tachikawa}.
  \end{proof}

It is instructive to compare the criterion that we just proved with the one
afforded by Theorem \ref{defs:cm2}:  The ring $S$ is Gorenstein if for
some essentially smooth{-}by{-}surjective factorization  $K\to P\to S$ of 
$\sigma$ of relative dimension $d$ the $S$-module  $\Ext nPSP$ is zero 
for $\grade_PS< n\le d$ and is projective for $n=\grade_PS$.

On the face of it, the difference lies only in the condition on the
$S$-module structure of a single, finitely generated  module $\Ext nPSP$.
However, this structure is induced through an inherently infinite
construction.  We elaborate below:

  \begin{remark}
    \label{action}
The action of any single \emph{element} $s\in S$ on every $\Ext iPSP$ 
can be computed from a free resolution of $S$ over $P$, as the result 
of applying the map
  \[
\CH i{\Hom R{s\id^F}P}\col\CH i{\Hom RFP}\lra\CH i{\Hom RFP}\,.
  \]
Each $F_j$ can be chosen finite free, so to compute such a map it
suffices to know the $i+1$ matrices with elements in $P$, describing
the differentials $\dd_j^F$ for $j=1,\dots, i+1$.

However, the $S$-module structure of $\Ext iPSP$ comes through an
isomorphism $\Ext iPSP\cong\CH i{\Hom PSI}$, where $I$ is an injective
resolution of $P$ over itself.  The information needed to construct
$I$ is not finite in two distinct ways: (1) the module $I^i$ is a direct 
sum of injective envelopes of $P/\fm$ for \emph{every} $\fm\in\Spec S$ 
of $P$ with $\Ext iP{P/\fm}P_\fm\ne0$, and for $1\le i<\dim P$ \emph{infinitely 
many} distinct $\fm$ satisfy this property; (2) injective envelopes are 
\emph{not finitely generated}, unless $\fm$ is minimal.
  \end{remark}

To finish, we take another look at Theorems~\ref{diagonal:gor} and 
\ref{portmanteau:Gor} this time against backdrops provided by results 
coming from three different directions: 

\noindent$\star$
Characterizations of regular local rings $(R,\fm,k)$ by the 
finiteness of $\pd_Rk$, or the vanishing of $\Ext nRkk$ for all 
(respectively, for some) $n>\dim R$;
see \cite{Ma}.

\noindent$\star$
Characterizations of Gorenstein local rings $(R,\fm,k)$ by the 
finiteness of $\id_RR$, or the vanishing of $\Ext nRkR$ for all 
(respectively, for some) $n>\depth R$; see \cite{Ma}.

\noindent$\star$
Characterizations of smooth $K$-algebras $S$ essentially of 
finite type by the finiteness of $\pd_{\env S}S$, or the vanishing 
of $\Ext n{\env S}SS$ for all $n\geq m$ (respectively, for $n\in[m, 
m+\dim S]$) when $\Omega_{S|K}$ can be generated by 
$m-1$ elements; see \cite{AI:mrl}.

Comparing the statements of these results one will
observe that the homological properties of a flat algebra of
finite type, viewed as a bimodule over itself, encode information about
all its fibers, and that the code is similar to the one that translates
properties of a local ring into homological data on its residue field.

This is the reasoning behind the conjecture stated in the introduction.

\end{document}